\documentclass[11pt,a4paper]{article}
\usepackage{amsmath,amssymb,amsthm,amsfonts,latexsym,graphicx,subfigure}
\usepackage[all,import]{xy}
\usepackage[numbers,sort&compress]{natbib}
\usepackage{indentfirst}
\usepackage{fancyhdr,enumerate}
\usepackage{bbm,color}

\usepackage{tikz}
\usetikzlibrary{topaths,calc}
\usetikzlibrary{positioning}
\tikzset{bluenode/.style={circle,fill=gray!50,minimum size=0.4cm,inner sep=0pt},}
\tikzset{rednode/.style={circle,fill=black!100,minimum size=0.4cm,inner sep=0pt},}

\usetikzlibrary{shapes.geometric, arrows}
\usepackage{accents,cases}
\usepackage{multirow}
\usepackage[lined, ruled, linesnumbered, longend]{algorithm2e}

\DeclareMathOperator{\id}{Id}

\usepackage{hyperref}

\theoremstyle{plain}
\newtheorem{theorem}{Theorem}[section]
\newtheorem{lemma}[theorem]{Lemma}
\newtheorem{proposition}[theorem]{Proposition}

\theoremstyle{definition}

\newtheorem{definition}[theorem]{Definition}
\newtheorem{ex}[theorem]{Example}
\theoremstyle{remark}
\newtheorem{remark}[theorem]{Remark}

\begin{document}
\bibliographystyle{unsrt} 
\title{ Signless Normalized Laplacian for Hypergraphs}
\author{ Eleonora Andreotti \thanks{Division of Vehicle Safety, Department of Mechanics and Maritime Sciences, Chalmers University of Technology, SE-412 96 Göteborg, Sweden \newline Email address: {\tt eleonora.andreotti@chalmers.se} (Eleonora Andreotti)} \and Raffaella Mulas \thanks{Max Planck Institute for Mathematics in the Sciences, Inselstr. 22, 04103 Leipzig, Germany.\newline
	Email address:
	{\tt r.mulas@soton.ac.uk} (Raffaella Mulas)}}
\date{}
\maketitle
\allowdisplaybreaks[4]
\begin{abstract}
The spectral theory of the normalized Laplacian for chemical hypergraphs is further investigated. The signless normalized Laplacian is introduced and it is shown that its spectrum for classical hypergraphs coincides with the spectrum of the normalized Laplacian for bipartite chemical hypergraphs. Furthermore, the spectra of special families of hypergraphs are established.

\vspace{0.2cm}


\vspace{0.2cm}

\end{abstract}


\section{Introduction}
In this work we bring forward the study of the normalized Laplacian that has been established for \textsl{chemical hypergraphs}: hypergraphs with the additional structure that each vertex in a hyperedge is either an input, an output or both (in which case we say that it is a catalyst for that hyperedge). Chemical hypergraphs have been introduced in \cite{Hypergraphs} with the idea of modelling chemical reaction networks and related ones, such as metabolic networks. In this model, each vertex represents a chemical element and each hyperedge represents a chemical reaction. Furthermore, in \cite{Master-Stability}, chemical hypergraphs have been used for modelling dynamical systems with high order interactions. In this model, the vertices represent oscillators while the hyperedges represent the interactions on which the dynamics depends.\newline

The spectrum of the normalized Laplacian $L$ reflects many structural properties of the network and several theoretical results on the eigenvalues have been established in \cite{Hypergraphs,Sharp,MulasZhang}. Furthermore, as shown in \cite{Sharp}, by defining the vertex degree in a way that it does not take catalysts into account, studying the spectrum of $L$ for chemical hypergraphs is equivalent to studying the spectrum of the \textsl{oriented hypergraphs} introduced in \cite{ReffRusnak} by Reff and Rusnak, in which catalysts are not included. Therefore, without loss of generality we can work on oriented hypergraphs. Here, in particular, we focus on the \textsl{bipartite} case and we show that the spectrum of the normalized Laplacian for bipartite chemical hypergraphs coincides with the spectrum of the \textsl{signless normalized Laplacian} that we introduce for classical hypergraphs. Furthermore, we establish the spectra of the signless normalized Laplacian for special families of such classical hypergraphs.\newline

Classical hypergraphs are widely used in various disciplines. For instance, they offer a valid model for transport networks \cite{andreotti2020eigenvalues}, neural networks (in whose context they are often called \textsl{neural codes}) \cite{neuro1,neuro2,neuro3,neuro5,neuro6,neuro7,neuro9,neuro10}, social networks \cite{ZhangLiu} and epidemiology networks \cite{BodoKatonaSimon}, just to mention some examples. It is worth noting that a simplicial complex $\mathcal{S}$ is a particular case of hypergraph with the additional constraint that, if a hyperedge belongs to $\mathcal{S}$, then also all its subsets belong to $\mathcal{S}$. Simplicial complexes are also widely present in applications. On the one hand, their more precise structure allows for a deeper theoretical study, compared to general hypergraphs. On the other hand, the constraints of simplicial complexes can be translated as constraints on the model, and this is not always convenient. Consider, for instance, a collaboration network that represents coauthoring of research papers: in this case, the fact that authors $A$, $B$ and $C$ have written a paper all together does not imply that $A$, $B$ and $C$ have all written single author papers, nor that $A$ and $B$ have written a paper together without $C$. In this case, a hypergraph would give a better model than a simplicial complex. \newline

\textbf{Structure of the paper.} In Section \ref{section:oriented} we introduce the basic definitions which are needed throughout the paper, while in Section \ref{section:twin} we introduce and discuss twin vertices. In Section \ref{section:bipartite} we prove new properties of bipartite oriented hypergraphs and we show that, from the spectral point of view, these are equivalent to classical hypergraphs with no input/output structure. Finally, in Section \ref{section:families} we investigate the spectra of new hypergraph structures that we introduce with the idea of generalizing well known graph structures.
\section{Basic definitions}\label{section:oriented}
\begin{definition}[\cite{ReffRusnak,Hypergraphs}]
				An \textbf{oriented hypergraph} is a pair $\Gamma=(\mathcal{V},\mathcal{H})$ such that $\mathcal{V}$ is a finite set of vertices and $\mathcal{H}$ is a set such that every element $h$ in $\mathcal{H}$ is a pair of disjoint elements $(h_{in},h_{out})$ (input and output) in $\mathcal{P}(\mathcal{V})\setminus\{\emptyset\}$. The elements of $\mathcal{H}$ are called the \textbf{oriented hyperedges}. Changing the orientation of a hyperedge $h$ means exchanging its input and output, leading to the pair $(h_{out},h_{in})$.
			\end{definition}

			   \begin{definition}
 Given $h\in \mathcal{H}$, we say that two vertices $i$ and $j$ are \textbf{co-oriented} in $h$ if they belong to the same orientation sets of $h$; we say that they are \textbf{anti-oriented} in $h$ if they belong to different orientation sets of $h$.
\end{definition}
			From now on, we fix a chemical hypergraph $\Gamma=(\mathcal{V},\mathcal{H})$ on $N$ vertices $v_1,\ldots,v_N$ and $M$ hyperedges $h_1,\ldots, h_M$. For simplicity, we assume that $\Gamma$ has no isolated vertices.
		\begin{remark}\label{remark:graphs1}
			Simple graphs can be seen as oriented hypergraphs such that $\# h_{in}=\# h_{out}=1$ for each $h\in \mathcal{H}$, that is, each edge has exactly one input and one output.
			\end{remark}
	\begin{definition}
	The \textbf{underlying hypergraph} of $\Gamma$ is $\Gamma':=(\mathcal{V},\mathcal{H}')$ where
	\begin{equation*}
	    \mathcal{H}':=\{(h_{in}\cup h_{out},\emptyset):h=(h_{in},h_{out})\in \mathcal{H}\}.
	\end{equation*}
	\end{definition}

\begin{definition}[\cite{Sharp}]
The \textbf{degree} of a vertex $v$ is
	\begin{equation*}
	    \deg(v):=\#\text{ hyperedges containing $v$.}
	\end{equation*}
	Similarly, the \textbf{cardinality} of a hyperedge $h$ is
	\begin{equation*}
	    \# h:=\#\{h_{in}\cup h_{out}\}.
	\end{equation*}
\end{definition}

\begin{definition}[\cite{Hypergraphs,MulasZhang}]
The \textbf{normalized Laplace operator} associated to $\Gamma$ is the $N\times N$ matrix
\begin{equation*}
    L:=\id -D^{-1}A,
\end{equation*}where $\id$ is the $N\times N$ identity matrix, $D$ is the \textbf{diagonal degree matrix} and $A$ is the \textbf{adjacency matrix} defined by $A_{ii}:=0$ for each $i=1,\ldots,n$ and
\begin{align*}
        A_{ij}:=& \# \{\text{hyperedges in which }v_i \text{ and }v_j\text{ are anti-oriented}\}+\\
        &-\# \{\text{hyperedges in which }v_i \text{ and }v_j\text{ are co-oriented}\}
\end{align*}for $i\neq j$. 
    \end{definition}
     We define the \textbf{spectrum of $\Gamma$} as the spectrum of $L$. As shown in \cite{Hypergraphs,MulasZhang}, this spectrum is given by $N$ real, nonnegative eigenvalues whose sum is $N$. We denote them by
     \begin{equation*}
 \lambda_1\leq\ldots\leq\lambda_N.
     \end{equation*}
    
\begin{definition}
We say that two vertices $v_i$ and $v_j$ are \textbf{adjacent}, denoted $v_i\sim v_j$, if they are contained at least in one common hyperedge.
\end{definition}

	\begin{remark}\label{remark:graphs2}
Consider a graph $\Gamma$ and let $\Gamma'$ be its underlying hypergraph. Then, the adjacency matrix $A$ of $\Gamma$ and the adjacency matrix $A'$ of $\Gamma'$ are such that $A'=-A$, while the degree matrices of $\Gamma$ and $\Gamma'$ coincide. Therefore, the normalized Laplacians of $\Gamma$ and $\Gamma'$ are
\begin{equation*}
    L=\id -D^{-1}A \qquad \text{ and }\qquad L'=\id +D^{-1}A=2\cdot \id -L,
\end{equation*}respectively. Hence, $\lambda$ is an eigenvalue for $L$ if and only if $2-\lambda$ is an eigenvalue for $L'$.
			\end{remark}

\begin{definition}
Let $\Gamma$ be an oriented hypergraph and let $\Gamma'$ be its underlying hypergraph. The \textbf{signless normalized Laplacian} of $\Gamma$ is the normalized Laplacian of $\Gamma'$.
\end{definition}

 \section{Twin vertices}\label{section:twin}
  \begin{definition}[\cite{MulasZhang}]
Two vertices $v_i$ and $v_j$ are \textbf{duplicate} if $A_{ik}=A_{jk}$ for all $k$. In particular, $A_{ij}=A_{ji}=A_{ii}=0$.
\end{definition}
 In \cite{MulasZhang} it is shown that $\hat{n}$ \textsl{duplicate vertices} produce the eigenvalue $1$ with multiplicity at least $\hat{n}-1$. Similarly, in this section we discuss \textsl{twin vertices}.
 \begin{definition}\label{def:twins}
We say that two vertices $v_i$ and $v_j$ are \textbf{twins} if they belong exactly to the same hyperedges, with the same orientations. In particular, $A_{ij}=-\deg (v_i)=-\deg (v_j)$ and $A_{ik}=A_{jk}$ for all $k\neq i,j$.
\end{definition}
\begin{remark}
While duplicate vertices are known also for graphs, twin vertices cannot exist for graphs, since in this case one assumes that each edge has one input and one output.
\end{remark}

We now generalize the notions of duplicate vertices and twin vertices by defining \textsl{duplicate families of twin vertices}.
\begin{definition}
Let $\Gamma=(\mathcal{V},\mathcal{H})$ be an oriented hypergraph. We say that a family of vertices $\mathcal{V}_1\sqcup\ldots\sqcup \mathcal{V}_l\subset \mathcal{V}$ is a \textbf{$l$-duplicate family of $t$-twin vertices} if
\begin{itemize}
    \item For each $i\in \{1,\ldots,l\}$, $\#\mathcal{V}_i=t$ and the $t$ vertices in $\mathcal{V}_i$ are twins to each other;
    \item For each $i,j\in \{1,\ldots,l\}$ with $i\neq j$, for each $v_i\in \mathcal{V}_i$ and for each $v_j\in \mathcal{V}_j$, we have that $A_{ij}=0$ and $A_{ik}=A_{jk}$ for all vertices $v_k$ that are not in the $l$-family, i.e. $v_k\in \mathcal{V}\setminus \mathcal{V}_1\sqcup\ldots\sqcup \mathcal{V}_l$.
\end{itemize}
\end{definition}
\begin{proposition}\label{prop:lt}
If $\Gamma$ contains a $l$-duplicate family of $t$ twins, then:
\begin{itemize}
    \item $t$ is eigenvalue with multiplicity at least $l-1$;
    \item $0$ is an eigenvalue with multiplicity at least $l(t-1)$.
\end{itemize}
\end{proposition}
\begin{proof}In order to show that $t$ is eigenvalue with multiplicity at least $l-1$, consider the following $l-1$ functions. For $i=2,\ldots,l$, let $f_i:\mathcal{V}\rightarrow \mathbb{R}$ such that $f_i:=1$ on $\mathcal{V}_1$, $f_i:=-1$ on $\mathcal{V}_i$ and $f_i:=0$ otherwise. Then,
\begin{itemize}
    \item For each $v_1\in \mathcal{V}_1$,
    \begin{equation*}
        Lf(v_1)=1-\frac{1}{\deg v_1}\sum_{v_1\neq v_j\in \mathcal{V}_1}-\deg v_1=1+t-1=t\cdot f(v_1);
    \end{equation*}
    \item For each $v_i\in \mathcal{V}_i$,
    \begin{equation*}
        Lf(v_i)=-1-\frac{1}{\deg v_i}\sum_{v_i\neq v_j\in \mathcal{V}_i}\deg v_i=-1-(t-1)=t\cdot f(v_i);
    \end{equation*}
    \item For each $v_k\in \mathcal{V}\setminus \mathcal{V}_1\sqcup\ldots\sqcup \mathcal{V}_l$,
    \begin{equation*}
        Lf(v_k)=-\frac{1}{\deg v_k}\left(\sum_{v_1\in \mathcal{V}_1}A_{1k}-\sum_{v_i\in \mathcal{V}_i}A_{ik} \right)=0=t\cdot f(v_k).
    \end{equation*}Therefore, $f_i$ is an eigenfunction for $t$. Furthermore, the functions $f_2,\ldots, f_l$ are linearly independent. Therefore, $t$ is an eigenvalue with multiplicity at least $l-1$.
\end{itemize}
Similarly, in order to prove that $0$ is eigenvalue with multiplicity at least $l(t-1)$, let $\mathcal{V}_i=\{v^i_1,\ldots,v^i_t\}$ and consider the $l(t-1)$ functions $g_j^i:\mathcal{V}\rightarrow\mathbb{R}$ defined as follows, for $i=1,\ldots,l$ and $j=2,\ldots,t$. Let $g_j^i(v_1^i):=1$, $g_j^i(v_j^i):=-1$ and $g_j^i:=0$ otherwise. Then, by \cite[Equation (5)]{Hypergraphs}, it is clear that each $g_j^i$ is an eigenfunction for $0$. Since, furthermore, these are $l(t-1)$ linearly independent functions, $0$ has multiplicity at least $l(t-1)$.
\end{proof}

\begin{proposition}\label{prop:twins}
If $\Gamma$ has $\hat{n}$ vertices that are twins to each other, $0$ is an eigenvalue with multiplicity at least $\hat{n}-1$. Furthermore, if $v_i$ and $v_j$ are twin vertices and $f$ is an eigenfunction for $L$ with eigenvalue $\lambda\neq 0$, then $f(v_i)=f(v_j)$.
\end{proposition}
\begin{proof}
The first claim follows from Proposition \ref{prop:lt}, by taking $t=1$.\newline

Now, assume that $v_i$ and $v_j$ are twin vertices and let $f$ be an eigenfunction for $L$ with eigenvalue $\lambda\neq 0$. Then,
\begin{equation*}
   \lambda f(v_i)= Lf(v_i)=f(v_i)+f(v_j)-\frac{1}{\deg v_i}\left(\sum_{k\neq i,j}A_{ik}f(v_k)\right)=Lf(v_j)=\lambda f(v_j).
\end{equation*}Since $\lambda\neq 0$, this implies that $f(v_i)=f(v_j)$.
\end{proof}
		\section{Bipartite hypergraphs}\label{section:bipartite}
\begin{definition}[\cite{Hypergraphs}\label{def:bipartite}]
	We say that a hypergraph $\Gamma$ is \textsl{bipartite} if one can decompose the vertex set as a disjoint union $\mathcal{V}=\mathcal{V}_1\sqcup \mathcal{V}_2$ such that, for every hyperedge $h$ of $\Gamma$, either $h$ has all its inputs in $\mathcal{V}_1$ and all its outputs in $\mathcal{V}_2$, or vice versa (Figure \ref{fig:bipartiteh}).
	\end{definition}

					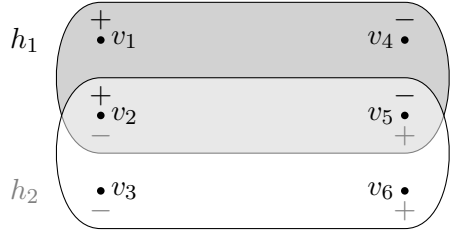
\begin{figure}[t!]
					\begin{center}
\begin{tikzpicture}
\node (v3) at (1,0) {};
\node (v2) at (1,1) {};
\node (v1) at (1,2) {};
\node (v6) at (5,0) {};
\node (v5) at (5,1) {};
\node (v4) at (5,2) {};

\begin{scope}[fill opacity=0.5]
\filldraw[fill=gray!70] ($(v1)+(0,0.5)$) 
to[out=180,in=180] ($(v2) + (0,-0.5)$) 
to[out=0,in=180] ($(v5) + (0,-0.5)$)
to[out=0,in=0] ($(v4) + (0,0.5)$)
to[out=180,in=0] ($(v1)+(0,0.5)$);
\filldraw[fill=white!70] ($(v2)+(0,0.5)$) 
to[out=180,in=180] ($(v3) + (0,-0.5)$) 
to[out=0,in=180] ($(v6) + (0,-0.5)$)
to[out=0,in=0] ($(v5) + (0,0.5)$)
to[out=180,in=0] ($(v2)+(0,0.5)$);
\end{scope}

\fill (v1) circle (0.05) node [right] {$v_1$} node [above] {\color{black}$+$};
\fill (v2) circle (0.05) node [right] {$v_2$} node [above] {\color{black}$+$} node [below] {\color{gray}$-$};
\fill (v3) circle (0.05) node [right] {$v_3$} node [below] {\color{gray}$-$};
\fill (v4) circle (0.05) node [left] {$v_4$} node [above] {\color{black}$-$};
\fill (v5) circle (0.05) node [left] {$v_5$} node [above] {\color{black}$-$}node [below] {\color{gray}$+$};
\fill (v6) circle (0.05) node [left] {$v_6$} node [below] {\color{gray}$+$};

\node at (0,2) {\color{black}$h_1$};
\node at (0,0) {\color{gray}$h_2$};
\end{tikzpicture}
					\end{center}
					\caption{A bipartite hypergraph with $\mathcal{V}_1=\{v_1,v_2,v_3\}$ and $\mathcal{V}_2=\{v_4,v_5,v_6\}$.}\label{fig:bipartiteh}
				\end{figure}
				
We now give the definition of \textsl{vertex-bipartite} hypergraph that, as we shall see in Lemma \ref{lemma:vertexbipartite} below, coincides with the definition of bipartite hypergraph.
				\begin{definition}
	We say that a hypergraph $\Gamma$ is \textbf{vertex-bipartite} if one can decompose the hyperedge set as a disjoint union $\mathcal{H}=\mathcal{H}_1\sqcup \mathcal{H}_2$ such that, for every vertex $v$ of $\Gamma$, either $v$ is an input only for hyperedges in $\mathcal{H}_1$ and it is an output only for hyperedges in $\mathcal{H}_2$, or vice versa.	\end{definition}
	\begin{lemma}\label{lemma:vertexbipartite}
	        Up to changing the orientation of some hyperedges, a hypergraph is bipartite if and only if it is vertex-bipartite. 
	\end{lemma}
	\begin{proof}Assume that $\Gamma$ is bipartite. Up to changing the orientation of some hyperedges, we can assume that the vertex set has a decomposition $\mathcal{V}=\mathcal{V}_1\sqcup \mathcal{V}_2$ such that each hyperedge $h$ has all its inputs in $\mathcal{V}_1$ and all its outputs in $\mathcal{V}_2$. Therefore, every vertex in $\mathcal{V}_1$ is an input only for hyperedges in $\mathcal{H}$, and every vertex in $\mathcal{V}_2$ is only an output for hyperedges in $\mathcal{H}$. It follows that the decomposition of the hyperedge set as $\mathcal{H}=\mathcal{H}\sqcup \emptyset$ gives a vertex-bipartition.\newline
	Now, assume that $\Gamma$ is vertex-bipartite, with $\mathcal{H}=\mathcal{H}_1\sqcup \mathcal{H}_2$. Assume, by contradiction, that $\Gamma$ is not bipartite. Then, up to changing the orientation of some hyperedges, there exist two vertices $v,w\in \mathcal{V}$ and two hyperedges $h_1,h_2\in \mathcal{H}$ such that:
	\begin{enumerate}
	    \item\label{item1} $h_1$ has both $v$ and $w$ as inputs;
	    \item\label{item2} $h_2$ has $v$ as input and $w$ as output.
	\end{enumerate}The fact that $v$ is an input in both $h_1$ and $h_2$ implies that $h_1$ and $h_2$ are in the same $\mathcal{H}_i$. On the other hand, the fact that $w$ is an input for $h_1$ and an output for $h_2$ implies that $h_1$ and $h_2$ do not belong to the same $\mathcal{H}_i$. This brings to a contradiction. Therefore, $\Gamma$ is bipartite.
	\end{proof}

\begin{proposition}\label{prop_bipartite}
If $\Gamma$ is bipartite, it is isospectral to its underlying hypergraph, therefore, in particular, also to every other bipartite hypergraph that has the same underlying hypergraph as $\Gamma$.
\end{proposition}
\begin{proof}Since $\Gamma$ is bipartite, up to switching (without loss of generality) the orientations of some hyperedges we can assume that all the inputs are in $\mathcal{V}_1$ and all the outputs are in $\mathcal{V}_2$, with $\mathcal{V}=\mathcal{V}_1\sqcup \mathcal{V}_2$. Furthermore, by Lemma 49 in \cite{Hypergraphs}, we can move a vertex from $\mathcal{V}_1$ to $\mathcal{V}_2$ or vice versa, by letting it be always an output or always an input, without affecting the spectrum. In particular, if we move all vertices to $\mathcal{V}_1$, we obtain the underlying hypergraph of $\Gamma$.
\end{proof}

\begin{remark}
As a consequence of Proposition \ref{prop_bipartite}, without loss of generality we can always assume that a bipartite hypergraph $\Gamma$ has only inputs, when studying the spectrum of the normalized Laplacian. In this case,
\begin{itemize}
    \item $A_{ij}=-\# \{h\in \mathcal{H}:v_i,v_j\in \mathcal{H}\}$ for each $i\neq j$;
    \item $\sum_jA_{ij}=-\sum_{h\ni v_i}(\# h-1)$, for each $v_i\in \mathcal{V}$.
\end{itemize}
\end{remark}

From here on we work on a hypergraph $\Gamma=(\mathcal{V},\mathcal{H})$ that has only inputs. Therefore, we focus on the signless normalized Laplacian of classical hypergraphs.

\section{Families of hypergraphs}\label{section:families}

\subsection{Hyperflowers}\label{section:flowers}
We now introduce and study \textsl{hyperflowers}: hypergraphs in which there is a set of nodes, the \textsl{core,} that is well connected to the other vertices, and a set of \textsl{peripheral nodes} such that each of them is contained in exactly one hyperedge. Hyperflowers are therefore a generalization of star graphs \cite{ANDREOTTI2018206}. 

\begin{definition}\label{def:hyperflowers}
A \textbf{$(l,r)$-hyperflower with t twins} (Figure \ref{fig:hyperflowertwins}) is an hypergraph $\Gamma=(\mathcal{V},\mathcal{H})$ whose vertex set can be written as $ \mathcal{V}=U\sqcup \mathcal{W}$, where:
\begin{itemize}
    \item $U$ is a set of $t\cdot l$ nodes $v_{11},\ldots,v_{1l},\ldots,v_{t1},\ldots,v_{tl}$ which are called \textbf{peripheral};
 \item There exist $r$ disjoint sets of vertices $h_1,\ldots,h_r \in \mathcal{P}(\mathcal{W})\setminus\{\emptyset\}$ such that
 $$\mathcal{H}=\{h|h=h_i\cup\bigcup_{z=1}^{t} v_{zj} \mbox{ for }i=1,\ldots,r \mbox{ and } j=1,\ldots,l \}.$$
\end{itemize}If $t=1$, we simply say that $\Gamma$ is a \textbf{$(l,r)$-hyperflower}.\newline
If $r=1$, we simply say that $\Gamma$ is a \textbf{$l$-hyperflower with t twins}.
\end{definition}
  	\begin{figure}[t!]
			\begin{center}
        \includegraphics[width=6cm]{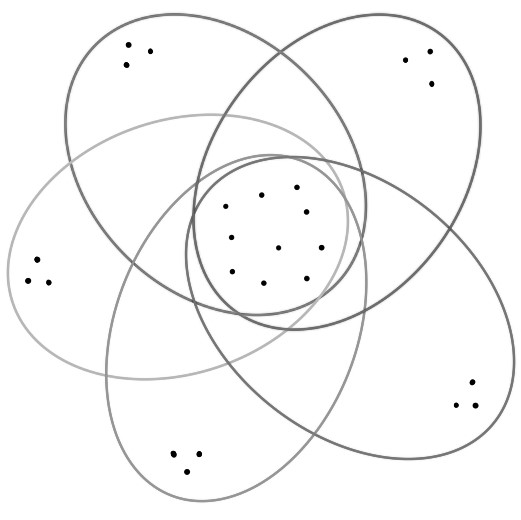}
        \caption{A $5$-hyperflower with $3$ twins.}        \label{fig:hyperflowertwins}
        \end{center}
    \end{figure}
\begin{remark}
The $(l,r)$-hyperflowers in Definition \ref{def:hyperflowers} are a particular case of the hyperstars in \cite{andreotti2020eigenvalues}, that also include weights and non-disjoint sets $h_1,\ldots,h_r$. Here we choose to study the particular structure of $(l,r)$-hyperflowers (and their generalizations with twins) because the strong symmetries of these structures allows for a deeper study of the spectrum. 
\end{remark}

\begin{proposition}\label{prop:l2hyperflower}
The spectrum of the $(l,2)$-hyperflower on $N$ nodes is given by:
\begin{itemize}
    \item $0$, with multiplicity $N-l-1$;
    \item $1$, with multiplicity $\geq l-1$;
    \item $\lambda_N>1$;
    \item $\lambda_{N-1}=N-\lambda_N-l+1\geq 1$.
\end{itemize}In the particular case in which $\#h$ is constant for each $h\in \mathcal{H}$, $\lambda_N=\frac{N-l}{2}+1$ and $\lambda_{N-1}=\frac{N-l}{2}$.
\end{proposition}
\begin{proof}
By \cite[Corollary 3.5]{MulasZhang}, $1$ is an eigenvalue with multiplicity at least $l-1$. Now, the $N-l$ vertices $v_{l+1},\ldots,v_N$ form two classes of twin vertices that generate the eigenvalue $0$ with multiplicity at least $N-l-2$. In particular, there exist $N-l-2$ linearly independent corresponding eigenfunctions $f_i:\mathcal{V}\rightarrow\mathbb{R}$ such that $f_i(v)=1$ for some $v\notin\{v_1,\ldots,v_l\}$, $f_i(w)=-1$ for a given $w$ twin of $v$, and $f_i=0$ otherwise. If we let $g(v_j):=1$ for each $j=1,\ldots,l$, $g(v'_1):=-1$ for exactly one $v'_1\in h_1$ and $g(v'_2):=-1$ for exactly one $v'_2\in h_2$, it's easy to see that $g$ is also an eigenfunction of $0$. Furthermore, the $f_i$'s and $g$ are all linearly independent, which implies that $0$ has multiplicity at least $N-l-1$.\newline

Now, by \cite[Theorem 3.1]{Sharp}, $\lambda_N\geq \frac{\sum_{h\in \mathcal{H}}\# h}{|\mathcal{H}|}>1$. We have therefore listed already $N-1$ eigenvalues and there is only one eigenvalue $\lambda$ missing. Since $\sum_{i=1}^N\lambda_i=N$, we have that $\lambda=N-\lambda_N-l+1$. In particular, since by \cite[Theorem 3.1]{Sharp} $\lambda_N\leq\max_{h\in \mathcal{H}}\#h$ with equality if and only if $\#h$ is constant, and $\max_{h\in \mathcal{H}}\#h\leq N-l$, we have that
\begin{equation*}
    \lambda=N-\lambda_N-l+1\geq 1,
\end{equation*}with equality if and only if $\#h$ is constant and equal to $N-l$, that is, if and only if $\#h_1=\#h_2=1$. Hence, $\lambda=\lambda_{N-1}$ and we have that $\lambda_{N-1}=1$ if and only if $\#h_1=\#h_2=1$.\newline

In general, if $\#h$ is constant for each $h\in \mathcal{H}$, then by \cite[Theorem 3.1]{Sharp} $\lambda_N=\#h=\frac{N-l}{2}+1$ and therefore $\lambda_{N-1}=\frac{N-l}{2}$.
\end{proof}

\begin{proposition}\label{prop:r1}
Let $\Gamma$ be an $(l,r)$--hyperflower with peripheral vertices $v_1,\ldots,v_l$. Let $\hat{\Gamma}:=(\hat{\mathcal{V}},\hat{\mathcal{H}})$ be the $(1,r)$--hyperflower defined by
\begin{equation*}
    \hat{\mathcal{V}}:=\mathcal{V}\setminus\{v_2,\ldots,v_l\} \qquad\text{and}\qquad \hat{\mathcal{H}}:=\{h\in \mathcal{H}: v_2,\ldots,v_l\notin h\}.
\end{equation*}Then, the spectrum of $\Gamma$ is given by:
\begin{itemize}
    \item The $N-l+1$ eigenvalues of $\hat{\Gamma}$, with multiplicity;
    \item $1$, with multiplicity at least $l-1$.
\end{itemize}
\end{proposition}
\begin{proof}
By \cite[Corollary 3.5]{MulasZhang}, adding $v_2,\ldots,v_l$ to $\hat{\Gamma}$ produces the eigenvalue $1$ with multiplicity $l-1$. Therefore, it is left to show that, if $\lambda$ is an eigenvalue of $\hat{\Gamma}$, then $\lambda$ is also an eigenvalue of $\Gamma$. Let $L$ and $A$ be the Laplacian and the adjacency matrix on $\Gamma$, respectively, and let $\hat{L}$ and $\hat{A}$ be the Laplacian and the adjacency matrix on $\hat{\Gamma}$, respectively. Let also $\hat{f}$ be an eigenfunction for $\hat{\Gamma}$ corresponding to the eigenvalue $\lambda$. Then,
\begin{equation*}
    \hat{L}\hat{f}(v_k)=\hat{f}(v_k)-\frac{1}{\deg_{\hat{\Gamma}}v_k}\sum_{v_i\in \hat{\mathcal{V}}\setminus \{v_k\}}\hat{A}_{ik}\hat{f}(v_i)=\lambda\cdot \hat{f}(v_k), \qquad\text{for all }v_k\in \hat{\mathcal{V}}.
\end{equation*}Now, let $f:\mathcal{V}\rightarrow \mathbb{R}$ be such that $f:=\hat{f}$ on $\hat{\mathcal{V}}$ and $f(v_2):=\ldots:=f(v_l):=\hat{f}(v_1)$.
Then,
\begin{align*}
     Lf(v_1)&=f(v_1)-\frac{1}{\deg v_1}\sum_{v_i\in \hat{\mathcal{V}}\setminus \{v_1\}}A_{i1}f(v_i)=\hat{f}(v_1)-\frac{1}{\deg_{\hat{\Gamma}} v_1}\sum_{v_i\in \hat{\mathcal{V}}\setminus \{v_1\}}\hat{A}_{i1}\hat{f}(v_i)\\&=\hat{L}\hat{f}(v_k) =\lambda\cdot \hat{f}(v_1) =\lambda\cdot f(v_1).
\end{align*}
  Similarly, for $j\in 2,\ldots,l$,
    \begin{equation*}
        Lf(v_j)=f(v_j)-\frac{1}{\deg v_j}\sum_{v_i\in \hat{\mathcal{V}}\setminus \{v_1\}}A_{ij}f(v_i)=\hat{f}(v_1)-\frac{1}{\deg_{\hat{\Gamma}} v_1}\sum_{v_i\in \hat{\mathcal{V}}\setminus \{v_1\}}\hat{A}_{i1}\hat{f}(v_i)=\lambda\cdot \hat{f}(v_1)=\lambda\cdot f(v_j).
    \end{equation*}Furthermore, for each $v_k\in \mathcal{V}\setminus\{v_1,\ldots,v_l\}$, we have that
    \begin{itemize}
        \item $\deg_{\hat{\Gamma}}(v_k)=1$ while $\deg(v_k)=l$;
        \item For each $v_{k'}\in \mathcal{V}\setminus\{v_1,\ldots,v_l,v_k\}$ such that $\hat{A}_{kk'}\neq 0$, $\hat{A}_{kk'}=-1$ while $A_{kk'}=-l$;
        \item $\hat{A}_{k1}=A_{k1}=-1$, and $A_{kj}=-1$ for each $j\in 2,\ldots, l$. 
    \end{itemize}Therefore, for for each $v_k\in \mathcal{V}\setminus\{v_1,\ldots,v_l\}$,
    \begin{align*}
        Lf(v_k)&=f(v_k)-\frac{1}{\deg v_k}\left(\sum_{k'}A_{kk'}f(v_{k'})+\sum_{j=1}^lA_{kj}f(v_j)\right)
        \\&=\hat{f}(v_k)-\frac{1}{l}\left(\sum_{k'}(-l)\hat{f}(v_{k'})+(-1)\sum_{j=1}^l\hat{f}(v_1)\right)
        \\&= \hat{f}(v_k)+\sum_{k'}\hat{f}(v_{k'})+\hat{f}(v_1)
        \\&=\hat{L}\hat{f}(v_k)=\lambda\cdot \hat{f}(v_k)=\lambda\cdot f(v_k).
    \end{align*}This proves that $\lambda$ is an eigenvalue for $L$, and $f$ is a corresponding eigenfunction.
\end{proof}
\begin{remark}
Proposition \ref{prop:r1} tells us that, in order to know the spectrum of a $(l,r)$--hyperflower, we can study the spectrum of the $(1,r)$--hyperflower obtained by deleting $l-1$ peripheral vertices and the hyperedges containing them, and then add $l-1$ $1$'s to the spectrum.
\end{remark}

\begin{proposition}The spectrum of the $l$-hyperflower with $t$ twins is given by:
\begin{itemize}
    \item $0$, with multiplicity $N-l$;
    \item $t$, with multiplicity $l-1$;
   \item $\lambda_N=N-tl+t$.\end{itemize}\end{proposition}
\begin{proof}Since all hyperedges have cardinality $N-tl+t$, by \cite[Theorem 3.1]{Sharp} we have that $\lambda_N=N-tl+t$. Furthermore, by Proposition \ref{prop:lt}, $t$ is an eigenvalue with multiplicity at least $l-1$. Since, clearly, $N-tl+t>t$, we have listed $l$ eigenvalues whose sum is $N$. It follows that $0$ has multiplicity $N-l$.
\end{proof}

\subsection{Complete hypergraphs}    
\begin{definition}[\cite{MulasZhang}]
We say that $\Gamma=(\mathcal{V},\mathcal{H})$ is the \textbf{$c$-complete hypergraph}, for some $c\geq 2$, if $\mathcal{V}$ has cardinality $N$ and $\mathcal{H}$ is given by all possible ${N \choose c}$ hyperedges of cardinality $c$.
\end{definition}
\begin{proposition}\label{prop:complete}
The spectrum of the $c$-complete hypergraph is given by:
\begin{itemize}
    \item $\frac{N-c}{N-1}$, with multiplicity $N-1$;
    \item $c$, with multiplicity $1$.
\end{itemize}
\end{proposition}
\begin{proof}By \cite[Theorem 3.1]{Sharp}, $\lambda_N=c$. Now, observe that each vertex $v$ has degree $d:={N-1 \choose c-1}$, while $a:=A_{ij}=-{N-2 \choose c-2}$ is constant for all $i\neq j$. Therefore, $\frac{a}{d}=-\frac{c-1}{N-1}$ and
\begin{equation*}
    Lf(v)=f(v)-\frac{a}{d}\left(\sum_{w\neq v}f(w)\right)=f(v)+\frac{c-1}{N-1}\left(\sum_{w\neq v}f(w)\right), \qquad \forall v\in \mathcal{V}.
\end{equation*}Now, for each $i=2,\ldots,N$, let $f(v_1):=1$, $f(v_i):=-1$ and $f:=0$ otherwise. Then,
\begin{itemize}
    \item $Lf(v_1)=1-\frac{c-1}{N-1}=\frac{N-c}{N-1}\cdot f(v_1)$,
    \item $Lf(v_i)=-1+\frac{c-1}{N-1}=\frac{N-c}{N-1}\cdot f(v_i)$, and
    \item $Lf(v_j)=0=\frac{N-c}{N-1}\cdot f(v_j)$ for all $j\neq 1,i$.
\end{itemize}Therefore, the $f_i$'s are $N-1$ linearly independent eigenfunctions for $\frac{N-c}{N-1}$. This proves the claim.
\end{proof}
\begin{ex}
Proposition \ref{prop:complete} tells us that the signless spectrum of the complete graph on $N$ nodes is given by $\frac{N-2}{N-1}$, with multiplicity $N-1$, and $2$ with multiplicity $1$. By Remark \ref{remark:graphs2}, this is equivalent to saying that the spectrum of the complete graph is given by $\frac{N}{N-1}$, with multiplicity $N-1$, and $0$ with multiplicity $1$. This is a well known result (see \cite{Chung}) and Proposition \ref{prop:complete} generalizes it.
\end{ex}

\subsection{Lattice Hypergraphs}
\textsl{Lattice graphs}, also called \textsl{grid graphs}, are well known both in graph theory and in applications \cite{Asllani2015, Andreotti2015ModelingTF,Lattice1,Lattice2,Lattice3,Lattice4,Lattice5,Lattice6}. For instance, they model topologies used in transportation networks, such as the Manhattan street network, and crystal structures used in  crystallography. These structures and their spectra are also widely used in statistical mechanics, in the study of ASEP, TASEP and SSEP models \cite{Mallick_2011, Speer1994, Schtz2017FluctuationsIS}, which have applications in the Ising model, (lattice) gas and which also describe the movement of ribosomes along the mRNA \cite{Gritsenko2015UnbiasedQM}. In this section we generalize the notion of lattice graph to the case of hypergraphs.
\begin{definition}
Given $l\in \mathbb{N}_{\geq 2}$, we define the \textbf{$l$-lattice} as the hypergraph $\Gamma=(\mathcal{V},\mathcal{H})$ on $l^2$ nodes and $2l$ hyperedges that can be drawn so that:
\begin{itemize}
    \item The vertices form a $l\times l$ grid, and
    \item The hyperedges are exactly the rows and the columns of the grid (Figure \ref{fig:lattice}).
\end{itemize}
\end{definition}
  	\begin{figure}[t!]
			\begin{center}
        \includegraphics[width=4cm]{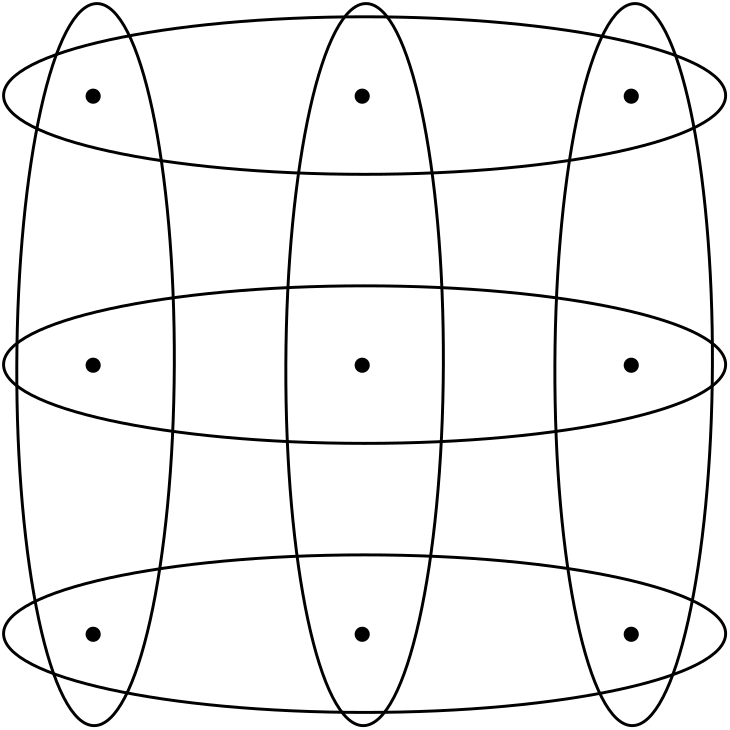}
        \caption{A $3$-lattice.}        \label{fig:lattice}
        \end{center}
    \end{figure}

\begin{proposition}
The spectrum of the $l$-lattice is given by:
\begin{itemize}
    \item $0$, with multiplicity $l^2-2l+1$;
    \item $\frac{l}{2}$, with multiplicity $2(l-1)$;
    \item $l$, with multiplicity $1$.
\end{itemize}
\end{proposition}
\begin{proof}By \cite[Theorem 3.1]{Sharp}, $\lambda_{l^2}=l$. Furthermore, by \cite[Corollay 33]{Hypergraphs}, since the maximum number of \textsl{linearly independent hyperedges} is $2l-1$, this implies that $0$ is an eigenvalue with multiplicity $l^2-2l+1$.\newline

Now, observe that $\deg v=2$ for each $v$ and
\begin{equation*}
    A_{ij}=\begin{cases}-1 &\text{if } v_i\sim v_j\\ 0 &\text{otherwise,}\end{cases}
\end{equation*}for all $i\neq j$. Therefore,
\begin{equation}\label{eq:Llattice}
Lf(v)=f(v)+\frac{1}{2}\left(\sum_{w\sim v}f(w)\right), \qquad \text{for all }v\in \mathcal{V}.
\end{equation}Fix a row of the $l$-lattice given by the vertices $w_1,\ldots,w_l$. For $i=1,\ldots,l-1$, let $f_i:\mathcal{V}\rightarrow\mathbb{R}$ be $1$ on the neighbors of $w_i$ with respect to the row, $-1$ on the neighbors of $w_i$ with respect to its column, and $0$ otherwise. Then, by \eqref{eq:Llattice}, it is easy to check that $f_i$ is an eigenfunction for $\frac{l}{2}$. Since the $f_i$'s are linearly independent, this proves the claim.

\end{proof}

\subsection{Hypercycles}
\begin{definition}
Fix $N$ and $l\in\{2,\ldots,\frac{N}{2}\}$. We say that $\Gamma=(\mathcal{V},\mathcal{H})$ is the \textbf{$l$-hypercycle} on $N$ nodes (Figure \ref{fig:cycle}) if $\mathcal{V}=\{v_1,\ldots,v_N\}$, $\mathcal{H}=\{h_1,\ldots,h_N\}$ and
\begin{equation*}
    h_i=\{v_i,\ldots,v_{i+l-1}\},
\end{equation*}where we let $v_{N+i}:=v_{i}$ for each $i=1,\ldots, N$.
\end{definition}
\begin{theorem}
The eigenvalues of the $l$-hypercycle are
\begin{equation*}
    \lambda_i=1+\frac{\sum_{r=1}^Nm(r)\cdot \cos \left(\frac{2\pi ir}{N}\right)}{l},\qquad \text{for }i=1,\ldots,N,
\end{equation*}where $m:\{0,\ldots,N\}\rightarrow\mathbb{Z}$ is such that:
\begin{itemize}
    \item $m(r):=l-r$ for all $r\in \{1,\ldots,l-1\}$
    \item $m(N-k):=m(k)=l-k$ for all $k\in \{1,\ldots,l-1\}$
    \item $m:=0$ otherwise.
\end{itemize}
\end{theorem}
\begin{proof}By construction, all vertices have degree $l$. Therefore, by \cite[Remark 2.17]{MulasZhang}, proving the claim is equivalent to proving that the eigenvalues of the adjacency matrix are
\begin{equation*}
    \mu_i=-\sum_{r=1}^Nm(r)\cdot \cos \left(\frac{2\pi ir}{N}\right),\qquad \text{for }i=1,\ldots,N.
\end{equation*}Observe that the adjacency matrix can be written as

\begin{equation*}A=-\begin{bmatrix}
0 & l-1 & l-2 &\ldots & 1 & 0 & \ldots & 0 & 1 & \ldots & l-2 &l-1\\
l-1 & 0 & l-1 & l-2 & \ldots & 1 & 0 & \ldots & 0 & 1 & \ldots & l-2\\
l-2 & l-1 & 0 & \ddots & \ddots & \ddots & \ddots & \ddots &  & \ddots & \ddots & \vdots\\
\vdots &  \ddots& \ddots & \ddots & \ddots & \ddots & \ddots & \ddots & \ddots & & \ddots & 1\\
1 & & & & & & & & & & & 0\\
0 & 1 & & & & & & & & & \ddots & \vdots\\
\vdots & 0 & \ddots & & & & & & & & \ddots & 0\\
0 &  & \ddots& & & & & & & & & 1\\
1 & 0 & & & & & & & & \ddots & \ddots& \vdots\\
\vdots & \ddots & \ddots& & & & & & \ddots & \ddots & l-1 & l-2\\
l-2 & & \ddots & \ddots & & \ddots & \ddots & & & l-1 & 0 & l-1\\
l-1 & l-2 & \ldots & \ldots & 0 & \ldots & 0 & 1 & \ldots & \ldots & l-1 & 0 \\
\end{bmatrix}\end{equation*}
Therefore,
\begin{equation*}A=-\begin{bmatrix}
m(0) & m(N-1) & m(N-2) &\ldots & m(1)\\
m(1) & m(0) & m(N-1)  &\ldots & m(2)\\
m(2) & m(1) & m(0) &\ldots & m(3)\\
\vdots & \vdots & \vdots & \ddots & \vdots\\
m(N-1) & m(N-2) & m(N-3) & \ldots & m(0)
\end{bmatrix}\end{equation*} where \begin{itemize}
    \item $m(r):=l-r$ for all $r\in \{1,\ldots,l-1\}$
    \item $m(N-k):=m(k)=l-k$ for all $k\in \{1,\ldots,l-1\}$
    \item $m:=0$ otherwise.
\end{itemize}
Hence, $A$ is a (symmetric) circulant matrix. By \cite{circulant}, the eigenvalues of $A$ are 
\begin{equation*}
    \mu_i=-\sum_{r=1}^Nm(r)\cdot \cos \left(\frac{2\pi ir}{N}\right),\qquad \text{for }i=1,\ldots,N.
\end{equation*}This proves the claim.
\end{proof}
 	\begin{figure}[t!]
			\begin{center}
        \includegraphics[width=5cm]{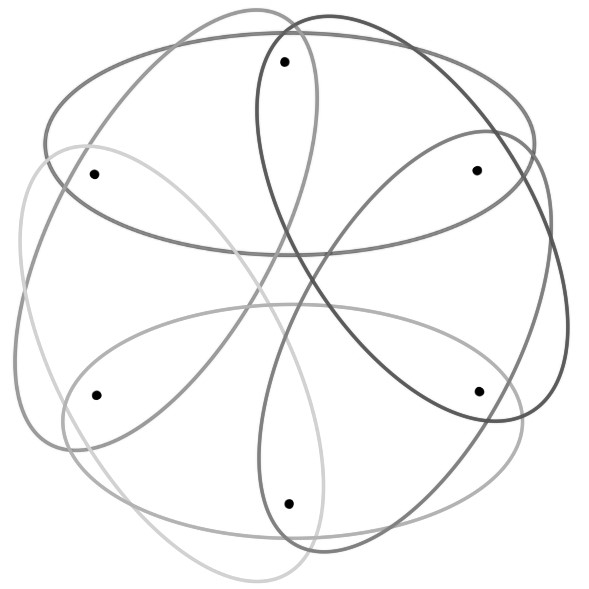}
        \caption{The $3$-hypercycle on $6$ nodes.}        \label{fig:cycle}
        \end{center}
    \end{figure}

\bibliography{Signless7.01.21}	

\begin{thebibliography}{10}

\bibitem{Hypergraphs}
J.~Jost and R.~Mulas.
\newblock {Hypergraph Laplace operators for chemical reaction networks}.
\newblock {\em Advances in Mathematics}, 351:870--896, 2019.

\bibitem{Master-Stability}
R.~Mulas, C.~Kuehn, and J.~Jost.
\newblock Coupled dynamics on hypergraphs: Master stability of steady states
  and synchronization.
\newblock {\em Phys. Rev. E}, 101:062313, 2020.

\bibitem{Sharp}
R.~Mulas.
\newblock {Sharp bounds for the largest eigenvalue of the normalized hypergraph
  Laplace Operator}.
\newblock {\em Mathematical Notes}, 2020.
\newblock To appear.

\bibitem{MulasZhang}
R.~Mulas and D.~Zhang.
\newblock {Spectral theory of Laplace Operators on chemical hypergraphs}.
\newblock arXiv:2004.14671.

\bibitem{ReffRusnak}
N.~Reff and L.~Rusnak.
\newblock {An oriented hypergraphic approach to algebraic graph theory}.
\newblock {\em Linear Algebra and its Applications}, 437:2262--2270, 2012.

\bibitem{andreotti2020eigenvalues}
E.~Andreotti.
\newblock Spectra of hyperstars on public transportation networks.
\newblock arXiv:2004.07831, 2020.

\bibitem{neuro1}
C.~Curto, V.~Itskov, A.~Veliz-Cuba, and N.~Youngs.
\newblock The neural ring: an algebraic tool for analyzing the intrinsic
  structure of neural codes.
\newblock {\em Bull. Math. Biol.}, 75(9):1571--1611, 2013.

\bibitem{neuro2}
C.~Giusti and V.~Itskov.
\newblock A no-go theorem for one-layer feedforward networks.
\newblock {\em Neural Comput.}, 26:2527--2540, 2014.

\bibitem{neuro3}
C.~Curto, E.~Gross, J.~Jeffries, K.~Morrison, Z.~Rosen M.~Omar, A.~Shiu, and
  N.~Youngs.
\newblock What makes a neural code convex?
\newblock {\em SIAM J. Appl. Algebra Geom.}, 1:222--238, 2016.

\bibitem{neuro5}
C.~Curto.
\newblock What can topology tell us about the neural code?
\newblock {\em Bull. Amer. Math. Soc.}, 54:63--78, 2017.

\bibitem{neuro6}
C.~Lienkaemper, A.~Shiu, and Z.~Woodstock.
\newblock Obstructions to convexity in neural codes.
\newblock {\em Adv. Appl. Math.}, 85:31--59, 2017.

\bibitem{neuro7}
M.~K. Franke and M.~Hoch.
\newblock Investigating an algebraic signature for max intersection-complete
  codes.
\newblock Texas A\&M Mathematics REU, 2017.

\bibitem{neuro9}
M.~K. Franke and S.~Muthiah.
\newblock Every neural code can be realized by convex sets.
\newblock {\em Adv. Appl. Math.}, 99:83--93, 2018.

\bibitem{neuro10}
R.~Mulas and N.M. Tran.
\newblock Minimal embedding dimensions of connected neural codes.
\newblock {\em Journal of Algebraic Statistics}, 11(1):99–106, 2020.

\bibitem{ZhangLiu}
Z.K. Zhang and C.~Liu.
\newblock A hypergraph model of social tagging networks.
\newblock {\em J. Stat. Mech.}, 2010(10):P10005, 2010.

\bibitem{BodoKatonaSimon}
{\'A}.~Bod{\'o}, G.Y. Katona, and P.L. Simon.
\newblock {SIS} epidemic propagation on hypergraphs.
\newblock {\em Bull. Math. Biol.}, 78(4):713--735, 2016.

\bibitem{ANDREOTTI2018206}
E.~Andreotti, D.~Remondini, G.~Servizi, and A.~Bazzani.
\newblock {On the multiplicity of Laplacian eigenvalues and Fiedler
  partitions}.
\newblock {\em Linear Algebra and its Applications}, 544:206 -- 222, 2018.

\bibitem{Chung}
F.~Chung.
\newblock Spectral graph theory.
\newblock {\em American Mathematical Society}, 1997.

\bibitem{Asllani2015}
M.~Asllani, D.M. Busiello, T.~Carletti, D.~Fanelli, and G.~Planchon.
\newblock Turing instabilities on cartesian product networks.
\newblock {\em Scientific Reports}, 5, 2015.

\bibitem{Andreotti2015ModelingTF}
E.~Andreotti, A.~Bazzani, S.~Rambaldi, N.~Guglielmi, and P.~Freguglia.
\newblock Modeling traffic fluctuations and congestion on a road network.
\newblock {\em Advances in Complex Systems}, 18, 2015.

\bibitem{Lattice1}
B.~D. Acharya and M.~K. Gill.
\newblock On the index of gracefulness of a graph and the gracefulness of
  two-dimensional square lattice graphs.
\newblock {\em Indian J. Math.}, 23:81--94, 1981.

\bibitem{Lattice2}
T.~Y. Chang.
\newblock {\em Domination Numbers of Grid Graphs}.
\newblock PhD thesis, Tampa, FL: University of South Florida, 1992.

\bibitem{Lattice3}
A.~Itai, C.~H. Papadimitriou, and J.~L. Szwarcfiter.
\newblock Hamilton paths in grid graphs.
\newblock {\em SIAM J. Comput.}, 11:676--686, 1982.

\bibitem{Lattice4}
H.~Iwashita, Y.~Nakazawa, J.~Kawahara, T.~Uno, and S.-I. Minato.
\newblock {Efficient Computation of the Number of Paths in a Grid Graph with
  Minimal Perfect Hash Functions}.
\newblock {\em TCS Technical Report. No. TCS-TR-A-13-64. Hokkaido University
  Division of Computer Science.}, 2013.

\bibitem{Lattice5}
V.~Reddy and S.~Skiena.
\newblock Frequencies of large distances in integer lattices.
\newblock {\em Technical Report, Department of Computer Science. Stony Brook,
  NY: State University of New York, Stony Brook}, 1989.

\bibitem{Lattice6}
T.~G. Schmalz, G.~E. Hite, and D.~J. Klein.
\newblock Compact self-avoiding circuits on two-dimensional lattices.
\newblock {\em J. Phys. A: Math. Gen.}, 17:445--453, 1984.

\bibitem{Mallick_2011}
K.~Mallick.
\newblock Some exact results for the exclusion process.
\newblock {\em Journal of Statistical Mechanics: Theory and Experiment},
  2011(01):P01024, 2011.

\bibitem{Speer1994}
E.R. Speer.
\newblock {\em The Two Species Totally Asymmetric Simple Exclusion Process},
  pages 91--102.
\newblock Springer US, Boston, MA, 1994.

\bibitem{Schtz2017FluctuationsIS}
G.~M. Sch{\"u}tz.
\newblock Fluctuations in stochastic interacting particle systems.
\newblock In {\em Part of the Springer Proceedings in Mathematics \& Statistics
  book series (PROMS, volume 282)}, 2017.

\bibitem{Gritsenko2015UnbiasedQM}
A.A. Gritsenko, M.~Hulsman, M.J.T. Reinders, and D.~de~Ridder.
\newblock Unbiased quantitative models of protein translation derived from
  ribosome profiling data.
\newblock {\em PLoS Computational Biology}, 11, 2015.

\bibitem{circulant}
J.~Montaldi.
\newblock {Notes on circulant matrices}.
\newblock Manchester Institute for Mathematical Sciences School of Mathematics,
  2012.

\end{thebibliography}

\end{document}